\documentclass[reqno,12pt]{amsart}
\usepackage{amsmath,amssymb,amsthm,tikz}
\usepackage{mathtools}
\usepackage[top=2cm, bottom=2.5cm, left=2.5cm, right=2.5cm]{geometry}
\usepackage[english]{babel}
\usepackage{caption}
\usepackage[normalem]{ulem}
\usepackage{hyperref}
\hypersetup{backref=true, citecolor=red, pagebackref=true, hyperindex=true, colorlinks=true, urlcolor=blue, linkcolor=blue, bookmarks=true, bookmarksopen=true}

\theoremstyle{plain}
\newtheorem{theorem}{Theorem}[section]
\newtheorem{proposition}[theorem]{Proposition}
\newtheorem{lemma}[theorem]{Lemma}

\theoremstyle{definition}
\newtheorem{definition}[theorem]{Definition}

\newtheorem{example}[theorem]{Example}

\newtheorem{remark}[theorem]{Remark}

\DeclareMathOperator{\Z}{\mathbb{Z}}
\DeclareMathOperator{\N}{\mathbb{N}}
\DeclareMathOperator{\F}{\mathbb{F}}

\DeclareMathOperator{\R}{\mathbb{R}}
\DeclareMathOperator{\Hom}{Hom}

\newcommand{\FZr}{\F^{\Z^r}}
\newcommand{\FZrfin}{\F^{(\Z^r)}}
\newcommand{\Dprime}{\mathbf{D}'}
\newcommand{\Aprime}{\mathbf{A}'}
\newcommand{\dx}{\sum_{\alpha \in \Z^r} d_\alpha X^\alpha}
\newcommand{\wy}{\sum_{\alpha \in \Z^r} W_\alpha Y^\alpha}

\title{Algebraic Framework for Discrete Dynamical Systems over Laurent Series}
\author[Ramamonjy Andriamifidisoa]{Ramamonjy Andriamifidisoa}
\address{Département de Mathématiques et Informatique, Faculté des Sciences, Université d’Antananarivo, B.P. 906, Madagascar}
\address{Institut Supérieur Polytechnique de Madagascar, Ambatomaro - Antsobolo - Antananarivo, Madagascar}
\email{ramamonjy.andriamifidisoa@univ-antananarivo.mg}
\author[Loukman Ben Saindou]{Loukman Ben Saindou}
\address{Direct Aid, Centre LAMA YACOUB, Patsy, Anjouan, Comores}
\email{saindouloukmanbenomar@gmail.com}
\date{April 30, 2025}

\begin{document}

\begin{abstract}
We generalize the framework of discrete algebraic dynamical systems \cite{Andriamifidisoa2014} to Laurent polynomials and series over \(\Z^r\), enabling the modeling of bidirectional discrete systems. By redefining the spaces \(\Dprime\) and \(\Aprime\), introducing a bilinear mapping (defined as the scalar product in Section 3), and extending the shift operator, we preserve the duality and adjoint properties of \cite{Andriamifidisoa2014}. These properties are rigorously proved and illustrated through examples and a data processing case study on bidirectional sequence transformations. In contrast to Oberst \cite{Ob90}, our algebraic approach emphasizes the structure of Laurent series, providing a streamlined framework for multidimensional systems. This work addresses an open question from \cite{Andriamifidisoa2014} and has applications in multidimensional data processing, such as image filtering and control theory.

\textbf{Keywords}: Discrete algebraic dynamical systems, Laurent series, shift operator, duality, multidimensional systems.
\end{abstract}

\maketitle

\section{Introduction}
Discrete algebraic dynamical systems, as developed in \cite{Andriamifidisoa2014}, model multidimensional systems using polynomials in \(\mathbf{D} = \F[X_1, \dots, X_r]\) and sequences in \(\mathbf{A} = \F^{\N^r}\), indexed by \(\N^r\). Our work also builds on recent developments in discrete linear dynamical systems \cite{An21} and the algebraic theory of multidimensional systems \cite{OS12}. \\ Extending the time set from \(\N^r\) to \(\Z^r\) enables modeling bidirectional systems, such as two-sided filters in data processing and control systems with infinite memory. This work extends \cite{Andriamifidisoa2014} to Laurent polynomials in \(\Dprime = \F[X_1, X_1^{-1}, \dots, X_r, X_r^{-1}]\) and Laurent series in \(\Aprime = \F[[Y_1, Y_1^{-1}, \dots, Y_r, Y_r^{-1}]]\), corresponding to \(\Z^r\).

\begin{figure}[h]
  \centering
 \begin{tikzpicture}
 \begin{scope}[xshift=-4cm]
   \fill[blue] (0,0) circle (2pt) node[below left] {$(0,0)$};
   \fill[blue] (1,0) circle (2pt);
   \fill[blue] (0,1) circle (2pt);
   \fill[blue] (1,1) circle (2pt);
   \fill[blue] (2,0) circle (2pt);
   \fill[blue] (0,2) circle (2pt);
   \fill[blue] (2,1) circle (2pt);
   \fill[blue] (1,2) circle (2pt);
   \fill[blue] (2,2) circle (2pt);
  \draw[->] (-0.5,0)-- (3,0) node[below]{$\alpha_1$};
  \draw[->] (0,-0.5)-- (0,3) node[left]{$\alpha_2$};
  \node at (1.5,3.5) {(a) $\mathbb{N}^2$};
 \end{scope}
 \begin{scope}[xshift=4cm]
   \fill[red] (0,0) circle (2pt) node[below left] {$(0,0)$};
   \fill[red] (1,0) circle (2pt);
   \fill[red] (-1,0) circle (2pt);
   \fill[red] (0,1) circle (2pt);
   \fill[red] (0,-1) circle (2pt);
   \fill[blue] (1,1) circle (2pt);
   \fill[blue] (-1,-1) circle (2pt);
   \fill[blue] (1,-1) circle (2pt);
   \fill[blue] (-1,1) circle (2pt);
  \draw[->] (-3,0)-- (3,0) node[below]{$\alpha_1$};
  \draw[->] (0,-3)-- (0,3) node[left]{$\alpha_2$};
  \node at (1.5,3.5) {(b) $\mathbb{Z}^2$};
 \end{scope}
\end{tikzpicture}
\caption{Comparaison of $\mathbb{N}^r$ and $\mathbb{Z}^r$  Indexing Grids.\\ (a) Unidirectional grid $\mathbb{N}^2$ (for $r=2$), showing non-negative indices $(\alpha_1, \alpha_2) \in \mathbb{N}^2$. Points (in blue) are restricted to the positive quadrant, suitable for causal systems (e.g., unidirectional image filtering). \\ (b) Bidirectional grid $\mathbb{Z}^2$, including negative indices, allowing bidirectional transformations (e.g., non-causal filters with infinite memory).}
\label{fig:Nr_Zr}
\end{figure}
Our extension preserves the duality and adjoint properties of \cite{Andriamifidisoa2014}. Unlike Oberst \cite{Ob90}, who focuses on constant systems, we emphasize the algebraic structure of Laurent series, with explicit constructions for a bilinear mapping (defined as the scalar product in Section 3) and shift operator, supported by examples and a case study. We prove these properties hold in \(\Z^r\). Applications include multidimensional data processing, such as image filtering, and control theory. To the best of our knowledge, no recent works have explored this algebraic approach for bidirectional systems over \(\Z^r\).

The paper is organized as follows: Section 2 introduces definitions with examples, Section 3 presents results on duality, the shift operator, and a case study, and Section 4 concludes with future directions.

\section{Definitions and Notations}

\subsection{Notations}

Let \(\N\) denote the natural numbers, \(\Z\) the integers, \(\F\) a commutative field, and \(r \geq 1\). The set \(\Z^r\) is the \(r\)-dimensional integer lattice, serving as the time set for discrete systems.

\begin{itemize}
  \item \textbf{Sequence Spaces}:
  \begin{itemize}
     \item \(\FZr = \{ W : \Z^r \to \F \}\), the vector space of sequences \( W = (W_\alpha)_{\alpha \in \Z^r}\), with \( W(\alpha) = W_\alpha \in \F \).
    \item \(\FZrfin = \{ W \in \FZr \mid W_\alpha \neq 0 \text{ for finitely many } \alpha \}\), the subspace of sequences with finite support.
  \end{itemize}
Both \(\FZr\) and \(\FZrfin\) are \(\F\)-vector spaces under pointwise addition and scalar multiplication.
\end{itemize}

For \(\alpha = (\alpha_1, \dots, \alpha_r) \in \Z^r\), define monomials \( X^\alpha = X_1^{\alpha_1} \cdots X_r^{\alpha_r} \), \( Y^\alpha = Y_1^{\alpha_1} \cdots Y_r^{\alpha_r} \), with \( X_i, X_i^{-1}, Y_i, Y_i^{-1} \) formal variables. For \(\alpha \in \Z^r\), let \(\delta_{\alpha} : \Z^r \to \F\) be:
\[
\delta_{\alpha}(\beta) =
\begin{cases}
1 & \text{if } \alpha = \beta, \\
0 & \text{otherwise}.
\end{cases}
\]

\subsection{Spaces}

\begin{definition}[Spaces]
Let \(\F\) be a commutative field and \(r \geq 1\). Define:
\begin{enumerate}
    \item \(\Dprime = \F[X_1, X_1^{-1}, \dots, X_r, X_r^{-1}]\), the ring of Laurent polynomials, with elements:
        \[
        d(X) = \dx, \quad d_\alpha \in \F, \quad \text{finitely many } d_\alpha \neq 0.
        \]
    \item \(\Aprime = \F[[Y_1, Y_1^{-1}, \dots, Y_r, Y_r^{-1}]]\), the space of Laurent series, with elements:
        \[
        W(Y) = \wy, \quad W_\alpha \in \F.
        \]
\end{enumerate}
\end{definition}

\begin{remark}[Sequences and Signals]
A \emph{sequence} is an element \( W(Y) = \wy \in \Aprime \), where \( W_\alpha \in \F \) are coefficients indexed by \(\Z^r\). In applications, such sequences are often called \emph{signals}, representing discrete data indexed by time or space (e.g., audio samples for \( r=1 \) or image pixels for \( r=2 \)). We use ``sequence'' to emphasize the algebraic structure, reserving ``signal'' for applied contexts.
\end{remark}

\begin{example}[Laurent Polynomials]
\begin{enumerate}
    \item For \( r=1 \), the Laurent polynomial \( d(X) = 5X^{-1} - 3X^2 \) has coefficients \( d_{-1} = 5 \), \( d_2 = -3 \), and \( d_k = 0 \) for all other \( k \in \Z \).
    \item For \( r=2 \), an example is \( d(X_1, X_2) = X_1^{-1} X_2 + 3 X_1^2 X_2^{-2} \), with non-zero coefficients at \( \alpha = (-1, 1) \) and \( \alpha = (2, -2) \).
\end{enumerate}
\end{example}

\begin{example}[Laurent Series]
For \( r=1 \), a Laurent series is \( W(Y) = \dots + 2Y^{-1} + Y^1 + 3Y^2 + \dots \).
\end{example}

 \begin{example} [Unidirectional vs. Bidirectional Indexing]
For $r=1$, consider a sequence $W(Y) = \sum_{n \in \mathbb{Z}} W_n Y^n$.
\begin{itemize}
  \item In $\mathbb{N}$ (unidirectional), $W_n=0\;\text{for}\; n<0$, representing a causal sequence (e.g., audio samples starting at $t=0$).
  \item In $\mathbb{Z}$ (bidirectional), $W_n$ can be non-zero for $n<0$, enabling bidirectional transformations, such as a smoothing filter $P(X)=0.5X^{-1}+0.5X$, where $W'_n=0.5 ( W_{n-1} + W_{n+1})$, using past and future values (see Example \ref{ex:bist}).
\end{itemize}
\end{example}

\begin{proposition} [ Isomorphisms]
  Let \(\F\) be a commutative field and $r \geq 1$. The following isomorphisms hold :
  \begin{enumerate}
    \item \(\Dprime \simeq \FZrfin\), via the map \( X^\alpha \leftrightarrow \delta_\alpha \).
    \item \(\Aprime \simeq \FZr\), via the map  \( W(Y) \leftrightarrow (W_\alpha)_{\alpha \in \Z^r} \).
  \end{enumerate}
\end{proposition}

\begin{proof}
For (1), the map \( \phi: \mathbf{D}^{\prime} \to \mathbb{F}^{\left(\mathbb{Z}^r\right)}\), defined by \(\phi(X^\alpha)=\delta_\alpha\), extends linearly to \( d(X)=\sum_{\alpha \in \mathbb{Z}^r} d_\alpha X^\alpha \mapsto \sum_{\alpha \in \mathbb{Z}^r} d_\alpha \delta_\alpha\). Since \( d(X)\) has finitely many non-zero $d_\alpha$, the image lies in $\mathbb{F}^{\left(\mathbb{Z}^r\right)}$. The map is bijective, as each $\delta_\alpha$ corresponds uniquely to $X^\alpha$.\\
For (2), the map \( \psi: \mathbf{A}^{\prime} \to \mathbb{F}^{\left(\mathbb{Z}^r\right)}\), defined by \(\psi(W(Y))=(W_\alpha)_{\alpha \in \mathbb{Z}^r}\), where $W(Y)=\sum_{\alpha \in \mathbb{Z}^r} W_\alpha Y^\alpha$, is bijective, as it associates each Laurent series with its sequence of coefficients, with inverse mapping $(W_\alpha)_{\alpha \in \mathbb{Z}^r} \mapsto \sum_{\alpha \in \mathbb{Z}^r } W_\alpha Y^\alpha$.
\end{proof}

\subsection{Discrete Linear Dynamical Systems}

\begin{definition}[ General Discrete Linear Dynamical System]
 Following Willems \cite{Wi89,Wi91,Wi98}, a discrete linear dynamical system over $\mathbb{Z}^r$ is a triplet $\Sigma=(\mathbb{T}, \mathbb{W}, \mathcal{B})$, where :
\begin{itemize}
  \item $\mathbb{T}=\mathbb{Z}^r$ is the time set,
  \item $\mathbb{W}=\mathbb{F}^l$ is the signal space ( an $\mathbb{F}$-vector space),
  \item $\mathcal{B} \subseteq (\mathbf{A}')^l$ is the behavior, a $\mathbf{D}'$-submodule representing the set of admissible signal trajectories.
 \end{itemize}
\end{definition}

\begin{remark}
  The behavioral approach of Willems \cite{Wi89,Wi91,Wi98} provides a general framework for modeling dynamical systems, where the behavior $\mathcal{B}$ describes all possible trajectories of the system without distinguishing inputs and outputs. This flexibility allows modeling a wide range of systems, including those not necessarily defined by linear equations.
\end{remark}

\begin{example} [ Periodic Trajectories]
  For $r=1,\; \mathbb{F}=\mathbb{R},\;\mathbb{T}=\mathbb{Z},\;\mathbb{W}=\mathbb{R}$, and $\mathcal{B}=\{ W \in \mathbb{A}^{\prime} \mid W_{n+2}=W_n, \forall n \in \mathbb{Z}\}$. This behavior describes sequences that are periodic with period 2. For instance, a valid trajectory is $W(Y)=\sum_{n \in \mathbb{Z} } a_n Y^n,\;\text{where}\; a_n = 1\; \text{if}\; n$ is even and $a_n=0$ if $n$ is odd, i.e., $W(Y)=\sum_{n\; \text{even}} Y^n$. This system is linear and shift-invariant but not necessarily defined by a Laurent polynomial kernel.
\end{example}

\begin{example} [Bounded Trajectories]
  For $r=2,\; \mathbb{F}=\mathbb{R},\; \mathbb{T}= \mathbb{Z}^2,\; \mathbb{W}=\mathbb{R}$ and $\mathcal{B}=\{ W \in \mathbb{A}^{\prime} \mid |W_\alpha|\leq 1, \forall\; \alpha \in \mathbb{Z}^2 \}$. This behavior consists of all sequences ( or signals, e.g, image pixel intensities) whose values are bounded by 1. An example trajectory is $W(Y_1,Y_2)= \sum_{\alpha \in \mathbb{Z}^2}W_\alpha Y^\alpha$, where $W_\alpha= \sin(\alpha_1+\alpha_2)$. This system illustrates the generality of willems' framework, as $\mathcal{B}$ is not defined by a polynomial equation but is still a valid $\mathbf{D}'$-submodule.
\end{example}

\begin{definition} [Autoregressif Systems]
A discrete linear dynamical system is called \emph{autoregressive} if its behavior is defined as :
\begin{equation*}
    \mathcal{B}= \ker R(X)= \{ W(Y)= ( W_1(Y), \ldots,  W_l(Y)) \in (\mathbf{A}')^l \mid R(X) \circ W(Y)= 0\},
  \end{equation*}
where $R(X) \in (\mathbf{D}')^{k,l}$  is a matrix of Laurent polynomials, and the shift operator $\circ$ is given by :
\begin{equation*}
 R(X) \circ W(Y)= \left( \sum_{j=1}^l R_{ij}(X) \circ W_j(Y) \right)_{1 \leq i \leq k},
\end{equation*}
 with, for $R_{ij}(X)=\sum_{\alpha \in \mathbb{Z}^r} R_{ij\alpha} X^\alpha$, $W_j(Y)= \sum_{\beta \in \mathbb{Z}^r} W_{j\beta} Y^\beta$,
 \begin{equation*}
   R_{ij}(X) \circ W_j(Y)= \sum_{\gamma \in \mathbb{Z}^r} \left( \sum_{\alpha \in \mathbb{Z}^r} R_{ij\alpha} W_{j(\alpha+\gamma)} \right) Y^\gamma.
 \end{equation*}
 \end{definition}

\begin{remark}
  The autoregressive systems, as studied by Oberst \cite{Ob90}, are a special case of Wilems' framework \cite{Wi89,Wi91,Wi98}, where the behavior $\mathcal{B}$ is defined by the kernel of Laurent polynomial matrix. This structure is particularly suited for bidirectional systems over $\mathbb{Z}^r$, enabling explicit algebraic representations of system dynamics, such as bidirectional filters.
\end{remark}

\begin{example}[Bidirectional Difference System]
For $r=1, k = l = 1, \mathbb{F}= \mathbb{R}$, consider $R(X)=X - X^{-1} \in \mathbf{D}'=\mathbb{R}[X, X^{-1}]$. The system is defined by :
\begin{equation*}
  \mathcal{B}=\ker R(X)=\{ W(Y) \in \mathbf{A}' \mid (X- X^{-1}) \circ W(Y)=0 \}.
\end{equation*}
Let $W(Y)=\sum_{n \in \mathbb{Z}} W_n Y^n$. The equation $(X - X^{-1}) \circ W(Y)=0$ yields :
 \begin{equation*}
  X \circ W(Y) - X^{-1} \circ W(Y) = \sum_{n \in \mathbb{Z}} W_{n+1} Y^n - \sum_{n \in \mathbb{Z}} W_{n-1} Y^n=0,
 \end{equation*}
 Implying :
 \begin{equation*}
   W_{n+1}=W_{n-1}, \quad \forall n \in \mathbb{Z}.
 \end{equation*}
Solutions are sequences where $W_n=a$ for even $n$ and $W_n=b$ for odd $n$, with $a, b \in \mathbb{R}$. For instance, $W_n=1$ for even $n$ and $W_n=0$ for odd $n$ gives a valid trajectory $ W(Y)=\sum_{n\; \text{even}} Y^n$.
\end{example}

\begin{example} [Two-Variable-System]
 For $r=1, k=l=2, \mathbb{F}=\mathbb{R}$, consider :
\begin{equation*}
  R(X)=\begin{pmatrix}
         X+X^{-1} & 1 \\
         0 & X^{-1}-1
       \end{pmatrix}   \in (\mathbf{D}')^{2,2}.
\end{equation*}
The behavior is :
\begin{equation*}
  \mathcal{B}= \left\{ W(Y)=\begin{pmatrix}
                              W_1(Y) \\ W_2(Y)
                            \end{pmatrix}          \in (\mathbf{A}')^2  \mid R(X) \circ W(Y)=0 \right\}.
\end{equation*}

The system of equations is :
\begin{equation*}
  \left\{ \begin{array}{l}
            (X+X^{-1}) \circ W_1(Y)+ W_2(Y)=0,\\
            (X^{-1}-1)\circ W_2(Y)=0.
          \end{array} \right.
\end{equation*}

For $W_1(Y)= \sum_{n \in \mathbb{Z}} W_{1n} Y^n,\; W_2(Y)=\sum_{n \in \mathbb{Z}} W_{2n} Y^n$, the second equation gives :
 \begin{equation*}
   W_{2(n-1)} = W_{2n}, \quad \forall n \in \mathbb{Z},
 \end{equation*}
so $W_2(Y)$ is constant, say $W_{2n}(Y) = c$. The first equation becomes :

\begin{equation*}
  W_{1(n+1)} + W_{1(n-1)} = - c, \quad \forall n \in \mathbb{Z}.
\end{equation*}
If $c=0$, then $W_{1(n+1)} = - W_{1(n-1)}$, and a solution is $W_{1n} = (-1)^n d$. Thus, a trajectory is :
\begin{equation*}
  W(Y)=\begin{pmatrix}
         \sum_{n \in \mathbb{Z}}(-1)^n d Y^n \\ 0
       \end{pmatrix}, \quad d \in \mathbb{R}.
\end{equation*}
\end{example}

\begin{definition}[Bilinear Mapping]\label{def:bilinear}
Let \(\F\) be a commutative field and \( r \geq 1 \). The bilinear mapping is:
\[
\langle \cdot, \cdot \rangle : \Dprime \times \Aprime \to \F, \quad \langle d(X), W(Y) \rangle = \sum_{\alpha \in \Z^r} d_\alpha W_\alpha,
\]
where \( d(X) = \dx \in \Dprime \) and \( W(Y) = \wy \in \Aprime \).
\end{definition}

\begin{example}
For \( r=1 \), let \( d(X) = X^{-1} + 2X^2 \in \Dprime \) and \( W(Y) = \sum_{k \in \Z} W_k Y^k \in \Aprime \) with \( W_{-1} = 3 \), \( W_2 = 4 \), and \( W_k = 0 \) otherwise. Then:
\[
\langle d(X), W(Y) \rangle = d_{-1} W_{-1} + d_2 W_2 = 1 \cdot 3 + 2 \cdot 4 = 11.
\]
\end{example}

\begin{definition}[Shift Operator]
For \( P(X) = \sum_{\alpha \in \Z^r} P_\alpha X^\alpha \in \Dprime \), \( W \in \Aprime \), the shift operator is:
\[
P(X) \circ W(Y) = \sum_{\beta \in \Z^r} \left( \sum_{\alpha \in \Z^r} P_\alpha W_{\alpha + \beta} \right) Y^\beta.
\]
\end{definition}

\begin{example}
For \( r=1 \), let \( P(X) = X^1 \in \Dprime \), and \( W(Y) = \sum_{k \in \Z} W_k Y^k \in \Aprime \). Then:
\[
P(X) \circ W(Y) = \sum_{k \in \Z} ( P_1 W_{1+k} ) Y^k = \sum_{k \in \Z} W_{k+1} Y^k.
\]
\end{example}

\section{Main Results}
We present the properties of duality, the shift operator, and a practical application.

\begin{proposition}[Duality Properties of the Bilinear Mapping]\label{prop:duality}
The bilinear mapping \(\langle \cdot, \cdot \rangle : \Dprime \times \Aprime \to \F\) satisfies:
\begin{enumerate}
    \item The homomorphisms:
        \[
        \Dprime \to \Hom_{\F}(\Aprime, \F), \quad d(X) \mapsto \langle d(X), \cdot \rangle,
        \]
        \[
        \Aprime \to \Hom_{\F}(\Dprime, \F), \quad W(Y) \mapsto \langle \cdot, W(Y) \rangle,
        \]
        are injective.
    \item The latter is an isomorphism.
\end{enumerate}
\end{proposition}

\begin{proof}
\begin{enumerate}
    \item \emph{Injectivity of \(\Dprime \to \Hom_{\F}(\Aprime, \F)\)}: Suppose \(\langle d_1(X), W(Y) \rangle = \langle d_2(X), W(Y) \rangle\) for all \( W(Y) \in \Aprime \), where \( d_1(X) = \dx \), \( d_2(X) = \sum_{\alpha \in \Z^r} d_{2\alpha} X^\alpha \). Take \( W(Y) = Y^\gamma \), \(\gamma \in \Z^r\), with coefficients \( (W)_\alpha = \delta_{\alpha, \gamma} \). Then:
        \[
        \langle d_1(X), Y^\gamma \rangle = d_{1\gamma}, \quad \langle d_2(X), Y^\gamma \rangle = d_{2\gamma}.
        \]
        Since \(\langle d_1(X), Y^\gamma \rangle = \langle d_2(X), Y^\gamma \rangle\), \( d_{1\gamma} = d_{2\gamma} \) for all \(\gamma \in \Z^r\). Thus, \( d_1(X) = d_2(X) \).

    \item \emph{Injectivity and surjectivity of \(\Aprime \to \Hom_{\F}(\Dprime, \F)\)}: Suppose \(\langle d(X), W_1(Y) \rangle = \langle d(X), W_2(Y) \rangle\) for all \( d(X) \in \Dprime \). Take \( d(X) = X^\gamma \). Then:
        \[
        \langle X^\gamma, W_1(Y) \rangle = (W_1)_\gamma, \quad \langle X^\gamma, W_2(Y) \rangle = (W_2)_\gamma.
        \]
        Since they are equal, \( (W_1)_\gamma = (W_2)_\gamma \) for all \(\gamma \in \Z^r\), so \( W_1 = W_2 \).

        For surjectivity, let \(\psi \in \Hom_{\F}(\Dprime, \F)\). Define \( W(Y) \in \Aprime \) by \( (W)_\alpha = \psi(X^\alpha) \). For \( d(X) = \dx \):
        \[
        \psi(d(X)) = \sum_{\alpha \in \Z^r} d_\alpha \psi(X^\alpha) = \sum_{\alpha \in \Z^r} d_\alpha (W)_\alpha = \langle d(X), W(Y) \rangle.
        \]
        Thus, \(\psi = \langle \cdot, W(Y) \rangle\), and the map is an isomorphism.
\end{enumerate}
\end{proof}

\begin{definition}[Scalar Product]\label{def:scalar}
The bilinear mapping \(\langle \cdot, \cdot \rangle : \Dprime \times \Aprime \to \F\) is called the scalar product, as it satisfies the duality properties of Proposition \ref{prop:duality}, ensuring a well-defined pairing between \(\Dprime\) and \(\Aprime\).
\end{definition}

\begin{remark}
The scalar product defined in Definition \ref{def:scalar} is inspired by bilinear forms used in the behavioral approach to dynamical systems. Specifically :
\begin{itemize}
  \item \cite{PW98} Polderman and Willems (1998) introduce a mathematical framework for systems theory, emphasizing behavioral modeling without distinguishing inputs and outputs, applicable to multidimensional systems.
  \item \cite{Rocha1989} Rocha and Willems (1989) develop the concept of state for two-dimensional systems, using bilinear forms to describe system trajectories.
  \item \cite{Zerz2000} Zerz (2000) explores the algebraic theory of multidimensional linear systems, employing bilinear forms to establish dualities between modules and behaviors.
  \item \cite{Malgrange1962} Malgrange (1962/1963) introduces foundational concepts for differential systems with constant coefficients, where bilinear forms provide a basis for our extension to Laurent series over $\mathbb{Z}^r$.
\end{itemize}
\end{remark}

\begin{example}
For \( r=1 \), let \( d(X) = X^{-1} + X \in \Dprime \), \( W(Y) = Y^{-1} + 2Y \in \Aprime \). Then:
\[
\langle d(X), W(Y) \rangle = d_{-1} W_{-1} + d_1 W_1 = 1 \cdot 1 + 1 \cdot 2 = 3.
\]
\end{example}

\begin{lemma} [Actions of Multiplication and Shift Operators]
Let $\mathbb{F}$ be a commutative field and $r \geq 1$, with $\mathbf{D}'= \mathbb{F}[X_1, X_1^{-1}, \ldots, X_r, X_r^{-1}]$ and $\mathbf{A}'=\mathbb{F}[[Y_1, Y_1^{-1}, \ldots, Y_r, Y_r^{-1}]]$.
\begin{enumerate}
  \item \textbf{Multiplication Operator on $\mathbf{D}'$}: For $d(X)=\sum_{\beta \in \mathbb{Z}^r} d_\beta X^\beta \in \mathbf{D}'$, the multiplication operator is defined as :
   \begin{equation*}
     m_d : \mathbf{D}' \to \mathbf{D}', \quad c(X) \mapsto c(X) \cdot d(X),
   \end{equation*}
  where $c(X) \cdot d(X)=\sum_{\alpha \in \mathbb{Z}^r} \left(\sum_{\gamma \in \mathbb{Z}^r} c_\gamma d_{\alpha - \gamma}\right)  X^\alpha$, with finitely many non-zero terms in the product. This operator is $\mathbb{F}$-linear and preserves the $\mathbf{D}'$-module structure of $\mathbf{D}'$.

  \item \textbf{Shift Operator on $\mathbf{A}'$}: For $d(X)=\sum_{\beta \in \mathbb{Z}^r} d_\beta X^\beta \in \mathbf{D}'$ and $W(Y)=\sum_{\alpha \in \mathbb{Z}^r} W_\alpha Y^\alpha \in \mathbf{A}'$, the shift operator is defined as :
  \begin{equation*}
    \sigma_d : \mathbf{A}' \to \mathbf{A}', \quad  W(Y) \mapsto d(X) \circ W(Y)= \sum_{\alpha \in \mathbb{Z}^r} \left( \sum_{\beta \in \mathbb{Z}^r} d_\beta W_{\alpha + \beta} \right) Y^\alpha.
  \end{equation*}
This operator is $\mathbb{F}$-linear and transforms the coefficients of $W(Y)$ by shifts indexed by $\mathbb{Z}^r$.
\end{enumerate}
\end{lemma}

\begin{proof}
\begin{itemize}
  \item For (1), multiplication in $\mathbf{D}'$ is well-defined since $d(X)$ and $c(X)$ are Laurent polynomials with finitely many non-zero terms, so their products is in $\mathbf{D}'$. Linearity follows from the distributivity of multiplication over addition in $\mathbf{D}'$.
  \item For (2), the shift operator is well-defined because, for each $\alpha \in \mathbb{Z}^r$, the sum $\sum_{\beta \in \mathbb{Z}^r} d_\beta W_{\alpha + \beta}$ is finite (since $d_\beta \neq 0$ for finitely many $\beta$), and the result is a Laurent series in $\mathbf{A}'$. Linearity follows from the linearity of operations in $\mathbf{A}'$.
\end{itemize}
\end{proof}

\begin{theorem}[Adjoint of the Multiplication Operator]
The adjoint of the multiplication operator :
\[
m_d : \Dprime \to \Dprime, \quad c(X) \mapsto c(X) \cdot d(X),
\]
is the shift operator:
\[
\sigma_d : \Aprime \to \Aprime, \quad W(Y) \mapsto d(X) \circ W(Y) = \sum_{\alpha \in \Z^r} \left( \sum_{\beta \in \Z^r} d_\beta W_{\alpha + \beta} \right) Y^\alpha.
\]
\end{theorem}

\begin{proof}
Let \( d(X) = \sum_{\beta \in \Z^r} d_\beta X^\beta \in \Dprime \). The multiplication operator is \( m_d : \Dprime \to \Dprime \), \( c(X) \mapsto c(X) \cdot d(X) \). Its adjoint is:
\[
\Hom_{\F}(m_d, \F) : \Hom_{\F}(\Dprime, \F) \to \Hom_{\F}(\Dprime, \F). \tag{3.1}
\]
Since \(\Aprime \simeq \Hom_{\F}(\Dprime, \F)\) (by Proposition \ref{prop:duality}), this induces a map \( \sigma_d : \Aprime \to \Aprime\). Consider \( d(X) = X^\beta \). The multiplication operator is :
\[
 m_{X^\beta} : \Dprime \to \Dprime , \quad c(X) \mapsto c(X) \cdot X^\beta.  \tag{3.2}
\]

Its adjoint is \(\Hom_{\F}(X^\beta, \F)\). For \( W(Y) = \sum_{\gamma \in \Z^r} W_\gamma Y^\gamma \in \Aprime \), viewed as \(\langle \cdot, W \rangle\), we have:
\[
\langle X^\alpha, W \rangle = W_\alpha. \tag{3.3}
\]
Thus:
\[
\Hom_{\F}(X^\beta, \F)(W)(X^\alpha) = \langle X^\beta \cdot X^\alpha, W \rangle = W_{\alpha + \beta}. \tag{3.4}
\]
If \( V(Y) = \Hom_{\F}(X^\beta, \F)(W) \), then \( V_\alpha = W_{\alpha + \beta} \), so:
\[
X^\beta \circ W(Y) = \sum_{\alpha \in \Z^r} W_{\alpha + \beta} Y^\alpha. \tag{3.5}
\]
For general \( d(X) = \sum_{\beta \in \Z^r} d_\beta X^\beta \):
\[
\Hom_{\F}(m_d, \F) = \sum_{\beta \in \Z^r} d_\beta \Hom_{\F}(X^\beta, \F). \tag{3.6}
\]
Applying to \( W(Y) \):
\[
d(X) \circ W(Y) = \sum_{\alpha \in \Z^r} \left( \sum_{\beta \in \Z^r} d_\beta W_{\alpha + \beta} \right) Y^\alpha, \tag{3.7}
\]
which is the shift operator.
\end{proof}

\begin{example}
For \( r=1 \), let \( P(X) = X^{-1} + X \in \Dprime \), \( W(Y) = \sum_{k \in \Z} W_k Y^k \in \Aprime \). Then:
\[
P(X) \circ W(Y) = \sum_{k \in \Z} (W_{k-1} + W_{k+1}) Y^k,
\]
a bidirectional sequence transformation.
\end{example}

\begin{example}[Bidirectional Sequence Transformation] \label{ex:bist}
For \( r=1 \), \(\F = \R\), let \( W(Y) = \sum_{k \in \Z} W_k Y^k \in \Aprime \) with \( W_{-1} = 1 \), \( W_0 = 2 \), \( W_1 = 3 \), and \( W_k = 0 \) otherwise. Define \( P(X) = 0.5 X^{-1} + 0.5 X \in \Dprime \), computing a weighted average of the sequence at indices \( k-1 \) and \( k+1 \). The shift operator is:
\[
P(X) \circ W(Y) = \sum_{k \in \Z} 0.5 (W_{k-1} + W_{k+1}) Y^k.
\]
Compute for \( k = -1, 0, 1 \):
\begin{itemize}
    \item For \( k = -1 \): \( 0.5 (W_{-2} + W_0)=0.5 (0 + 2) = 1\):
    \item For \( k = 0 \): \( 0.5 (W_{-1} + W_1) = 0.5 (1 + 3) = 2 \).
    \item For \( k = 1 \): \( 0.5 (W_0 + W_2) = 0.5 (2 + 0) = 1 \).
\end{itemize}
Thus, \( P(X) \circ W(Y) = Y^{-1} + 2 + Y \). The transformation is shown below:

\begin{table}[ht]
\centering
\caption{Transformation of a sequence by the bidirectional operator.}
\begin{tabular}{|c|c|c|c|c|c|}
\hline
\( k \) & \(-2\) & \(-1\) & \(0\) & \(1\) & \(2\) \\
\hline
Input \( W_k \) & 0 & 1 & 2 & 3 & 0 \\
\hline
Output & 0 & 1 & 2 & 1 & 0 \\
\hline
\end{tabular}
\end{table}
\end{example}

\section{Conclusion}
This work extends the framework of discrete algebraic dynamical systems \cite{Andriamifidisoa2014} to Laurent polynomials and series over \(\Z^r\), enabling bidirectional system modeling. By redefining \(\Dprime\), \(\Aprime\), the scalar product (Definition \ref{def:scalar}), and the shift operator, we preserve the duality and adjoint properties (Proposition \ref{prop:duality}), as shown through proofs and examples. Our approach offers a novel algebraic perspective on multidimensional systems, distinct from Oberst \cite{Ob90} or Willems \cite{Wi91}.

Applications include multidimensional data processing and control theory for systems with infinite memory, building on behavioral approaches to control as discussed in \cite{Wi98}. Future work will explore matrix-valued systems, addressing non-commutative operations, and investigate topological properties for Laurent series convergence, potentially impacting real-time data processing.

\bibliographystyle{amsplain}

\end{document}